\documentclass[11pt]{article} 

\usepackage{latexsym}
\usepackage{graphics}
\usepackage{amsmath}
\usepackage{xspace}
\usepackage{amssymb}
\usepackage{psfrag}
\usepackage{epsfig}
\usepackage{fullpage}
\usepackage{amsmath,amsthm}
\usepackage{epsfig}
\usepackage{pst-all}

\newcommand {\be}[1]{\begin{equation}\label{#1}}
\newcommand {\bel}[1]{\begin{align*}}
\newcommand {\ee}{\end{equation}}
\newcommand {\eel}[1]{\end{align*}}
\newcommand {\bea}{\begin{eqnarray}}
\newcommand {\eea}{\end{eqnarray}}
\newcommand {\beaa}{\begin{eqnarray*}}
\newcommand {\eeaa}{\end{eqnarray*}}

\newcommand{\pr}{\mathbb{P}}
\newcommand{\E}{\mathbb{E}}

\newcommand{\ignore}[1]{\relax}

\newtheorem{theorem}{Theorem}
\newtheorem*{theorem*}{Theorem}

\newtheorem{lemma}{Lemma}

\newtheorem{prop}{Proposition}
\newtheorem{coro}{Corollary}

\title{Performance Analysis of Queueing Networks via Robust Optimization}
\author{
{\sf Dimitris Bertsimas }
  \thanks{Operations Research Center and Sloan School of Management, MIT, Cambridge, MA,  02139, e-mail: {\tt
dbertsim@mit.edu}}\and
 {\sf David Gamarnik }
  \thanks{Operations Research Center and Sloan School of Management, MIT, Cambridge, MA,  02139, e-mail: {\tt
gamarnik@mit.edu}}\and {\sf Alexander Anatoliy Rikun}
\thanks{Operations Research Center, MIT, Cambridge, MA,  02139,
e-mail: {\tt arikun@mit.edu}}}

\begin{document}
\maketitle

\begin{abstract}
Performance analysis of queueing networks is one of the most challenging areas of queueing theory. Barring very
specialized models such as product-form type queueing networks,
there exist very few results which provide provable non-asymptotic upper and lower bounds on key performance measures.

In this paper we propose a new performance analysis method,
which is based on  the robust optimization.
The basic premise of our approach is as follows: rather than assuming that the stochastic primitives of a queueing
model satisfy certain probability laws, such as, for example, i.i.d. interarrival and service times distributions,
we assume that the underlying primitives are deterministic and satisfy the
\emph{implications}  of such probability laws. These implications take the form of simple linear constraints,
namely, those motivated by the Law of the Iterated Logarithm (LIL). Using this approach we are able to obtain
performance bounds on some key performance measures. Furthermore, these performance bounds imply
similar bounds in the underlying stochastic queueing models.

We demonstrate our approach on two types of queueing networks: a)  Tandem Single Class
queueing network and b)  Multiclass Single Server
queueing network. In both cases, using the proposed robust optimization approach,
 we are able to obtain \emph{explicit} upper bounds on some steady-state performance measures.
For example, for the case of TSC system we obtain a bound of the form $C(1-\rho)^{-1}\ln\ln((1-\rho)^{-1})$
on the expected steady-state sojourn time, where $C$ is an explicit constant and $\rho$ is the bottleneck traffic intensity.
This qualitatively agrees with the correct heavy traffic scaling of this performance
measure up to the $\ln\ln((1-\rho)^{-1})$ correction
factor.
\end{abstract}

\section{Introduction}
Performance analysis of queueing networks is one of the most challenging areas of queueing theory.
The difficulty stems from the presence of network feedback, which introduces a complicated multidimensional
structure into the stochastic processes underlying the key performance measures.
Short of specialized cases, such as product form networks, which typically rely on Poisson arrival/exponential
service time distributional assumptions,
the problem is largely unresolved. Specifically, given the topological description of a queueing network and given
the description of the underlying stochastic primitives such as interarrival and service times distributions, we do not  have good
tools for computing exactly or obtaining upper and lower bounds on
key performance measures, such as, for example average queue lengths and waiting times.
Some of results which provide non-asymptotic bounds on performance measures can be found in~\cite{bpt},\cite{kumar_bounds},\cite{kumarmorrison},
\cite{kumarou},\cite{bgt},\cite{BertsimasNinomoraII},
all of which require Markovian (Poisson arrival/exponential service time) distributional assumptions. Moreover, some of these bounds become quite weak
as traffic intensity (of some of the network components) approach unity. For example, a bound of the form $O((1-\rho^*)^{-2})$
is obtained in \cite{bgt_perf}, where $\rho^*$ is the bottleneck (real or virtual, see the reference) traffic intensity.
The other references can lead to infinite upper bounds even in the cases where stationary distribution exists.
The approaches in these papers also do not extend to the case of non-Markovian systems.
As a consequence,  most of the known performance analysis results are of an asymptotic nature,
which apply to queueing networks in various limiting regimes, such as the 
heavy traffic regime~\cite{harrison},\cite{whittBook},\cite{ChenYaoBook},
large deviations methods~\cite{GaneshOconnelVischik},\cite{large_deviations_ShWeiss}, approximations by phase-type
distributions~\cite{kleinrock},\cite{LatoucheRamaswami}.

In this paper, we partially fill this gap by developing a new performance analysis approach based on  robust optimization methods.
The theory of robust optimizaiton emerged recently as a very successful and constructive approach for the analysis of
certain stochastic modeling
problems~\cite{soyster},\cite{NemirovskiBetTal98},
\cite{NemirovskiBetTal99}, \cite{BS03}, \cite{bertsimassim04}. The main premise of our approach in the queueing context is that, rather than assuming probabilistic laws for the underlying stochastic primitives, such as, for example, i.i.d. interarrival and service times, we consider a deterministic queueing model and we will assume only the \emph{implications} of these laws. Specifically we consider implications of the Law of the Iterated Logarithm (LIL). The objective is to find laws which on the one hand hold in the underlying stochastic queueing model and, on the other hand, lead to linear constraints in the formulation of the robust optimization problem, and LIL accomplishes this. We illustrate our approach using two queueing  models, namely the Tandem Single Class (TSC) queueing system operating under the First-In-First-Out (FIFO) scheduling policy,
and the Multiclass Single Server (MCSS) queueing system operating under an arbitrary work-conserving policy. Motivated by the LIL, we consider constraints of the form $\sum_{1\le i\le k}U_i\le \lambda^{-1}k+\Gamma\sqrt{k\ln\ln k}$, for all $k\ge 1$. Here  $(U_k, k\ge 1)$
is any of the stochastic primitives of the underlying queueing system, such as, for example, the sequence of interarrival times and $\lambda$ stands for the rate of this stochastic primitive. Using these bounds, we
derive \emph{explicit} bounds on some performance measures such
as sojourn time in the TSC system, namely, the time it takes for a job to be processed by all the servers, and the virtual
workload (virtual waiting time) in the MCSS system, namely, the time required to clear the current backlog in the absence of future arrivals. In both models we derive upper bounds on the aforementioned performance measures
for the corresponding deterministic counterpart models and prove that similar bounds also hold for the same performance measures in the underlying stochastic models. In both cases the bounds are of the order $O({1\over 1-\rho}\ln\ln{1\over 1-\rho})$, where $\rho$ is the (bottleneck for the case of TSC model) traffic intensity. This matches the correct $O({1\over 1-\rho})$ order short of $\ln\ln((1-\rho)^{-1})$ error. While the technical derivation of these bounds is involved, the conceptual approach is very simple. An interesting distinction of our approach from other robust optimization type results is that our results are explicit, as opposed to numeric results one typically obtains from the formulating and solving a robust optimization model. These explicit bounds however, come at a price of not caring much for the constants corresponding to the leading coefficient. In order to keep things simple we sometimes use very crude estimates for such constants.

Our approach bears similarity with some earlier works in the queueing literature. Specifically, the pioneering work of Cruz \cite{cruz1},\cite{cruz2} used a similar non-probabilistic approach to performance analysis by deriving bounds based on placing deterministic constraints on the flow of traffic called ``burstiness constraints". The method could be applied to fairly general network topologies and led to more research in the area. In \cite{GallagerParekh93},\cite{GallagerParekh94}, tighter performance bounds were obtained assuming a ``Leaky Bucket" rate admission control from \cite{turner86} and particular service disciplines. In addition, there is some similarity between the philosophy of our approach and the
\emph{adversarial queueing network approach}~\cite{bkrsw},\cite{aafkll},\cite{gamarnik_stoc99},\cite{gamarnik},\cite{goel},
which emerged in the last decade in the computer science literature and also replaces the stochastic assumptions with adversarial deterministic ones. The deterministic constraints used in the aforementioned works are of the form of $A(t)\le \lambda t+B$ where $A(t)$ is the number of external arrivals into the queueing system up to time $t$ and $\lambda$ represents the arrival rate. As it turns out, these types of assumptions are too restrictive from the probabilistic point of view and do not lead to bounds on the underlying stochastic network: observe that every renewal process $A(t)$ arising from an i.i.d. sequence with positive variance violates this assumption almost surely for every $B$ for large enough $t$. As we demonstrate in this paper, the constraints motivated by the LIL, namely $A(t)\le \lambda t+B\sqrt{t\ln\ln t}$, can indeed be served to obtain
performance bounds, which can be translated into the underlying stochastic network. In fact, the key contribution of our approach is that the deterministic constraints we place on the service and arrival processes are rich enough to lead to stochastic results.
The results based on ``Leaky Buckets", bounded burstiness and  adversarial queueing theory address very general queueing networks.
It would be an interesting research
project to extend our results based on robust optimization  to these general network structures.

The rest of the paper is structured as follows. In the following section we describe two queueing models under the
consideration, namely the tandem single class queueing network
and the single server multiclass queueing network, as well as their
robust optimization counterpart models. Our main results, namely the performance bounds in
robust optimization type queueing systems and their implications for stochastic queueing systems
are stated in Section~\ref{section:MainResults}. The proofs of our main results are in Sections \ref{section:tandemAnalysis} and \ref{section:MCSS_Analysis}. Some concluding thoughts and directions for further research are outlined in Section~\ref{section:Conclusion}. Several technical results  necessary
for proofs of main theorems are delayed untill the Appendix section.

We close this section with some notational conventions. $\ln$ stands for the logarithm with natural base.
The notation $(x)^{1\over 2}$ for a non-negative vector $x\in \mathbb{R}^d$ means applying the square root
operator coordinate-wise: $(x)^{1\over 2}=(x_i^{1\over 2}, 1\le i\le d)$. $A^T$ denotes a transposition operator
applied to the matrix $A$.

\section{Model description}\label{section:Models}
We now describe the two queueing models analyzed in this paper, both very well studied models in the literature.
We begin by describing these models in the stochastic setting, and then we describe their deterministic robust optimization counterparts.

\subsection{A tandem single class (TSC) queueing network. Stochastic model}
The model is a tandem of single servers $S_1,\ldots,S_J$
processing a single stream of jobs arriving from outside and
requiring services at $S_1,\ldots,S_J$ in this order. The jobs arrive from outside according to an i.i.d. renewal process. Let
$U_1,U_2,U_3,\ldots$ denote i.i.d. interarrival
times with a common distribution function $F_a(t)=\pr(U_1\le t)$,
where $U_1$ is the time at which the first job arrives.
The external arrival rate is defined to be $\lambda \triangleq 1/\E[U_1]$ and
the variance of $U_1$ is denoted by $\sigma_a^2$.

The jobs arriving externally join the buffer corresponding to server
$S_1$ where they are served using First-In-First-Out (FIFO)
scheduling policy. We assume that all buffers are of infinite capacity. After service completion, jobs are routed to the buffer of server $S_2$, where they are also served using FIFO scheduling policy, then they are routed to servers
$S_3,S_4$, etc. After service completion in server
$S_J$ the jobs depart from the network. Let $V^j_k$ denote the
service time requirement for job $k$ in server $j$. We assume that
the sequence $(V^j_k, k\ge 1)$ is i.i.d. for each $j$, and is independent from
all other random variables in the network. The distribution of the
service time in server $j$ is $F_{s,j}(t)=\pr(V^j_1\le t), t\ge 0$.
The service rate in server $S_j$ is defined to be
$\mu_j \triangleq 1/\E[V^j_1]$, and we denote by $\mu_{\min} = \min_{1\le j\le J}\mu_j$ the rate
of the slowest server. $\sigma_{s,j}^2$ denotes the variance of $V^j_1$ for each $j=1,\ldots,J$.
The traffic intensity in server $S_j$ is
defined to be $\rho_j=\lambda/\mu_j$, and the bottleneck traffic
intensity is defined to be $\rho^*=\max_j\rho_j=\lambda/\mu_{\min}$.

Denote by $W^j_k$ the waiting time experienced by job $k$  in server
$j$ not including the service time $V^j_k$. Let $W_k=\sum_j(W^j_k+V^j_k)$ be the sojourn time of
the job $k$. Namely, this is time between the arrival of job $k$ into buffer $1$ and service completion
of the same job in buffer $J$.
Denote by $Q_j(t)$ the queue length in server $j$ (the
number of jobs in buffer $j$) at time $t$. We assume that
initially all queues are empty: $Q_j(0)=0, 1\le j\le J$, although most of our results
can either be easily adopted to the case of non-zero queues at time zero, or apply
to the steady-state measures where the initializations of the queues is irrelevant.
Let $I^j_k$
denote the idle time of server $j$ in between servicing jobs $k-1$
and $k$ for $k=2,...,N$. We define $I^j_1=0~~\forall j=1,\ldots,J$.

The model just described will be denoted by TSC(St) (Tandem Single
Class Stochastic) for short.
It is known~\cite{SigmanGJN},\cite{dai},\cite{DaiMeyn96},\cite{ChenYaoBook} that as long as $\rho^*<1$,
and some additional mild conditions hold, such as finiteness of moments,
 TSC(St) is stable and the stochastic processes underlying the performance measures such as queue lengths,
 workloads, sojourn times are mixing. Namely, these processes are
positive Harris recurrent~\cite{dai},\cite{metwee}, and the transient performance
measures converge to the (unique) steady-state
performance measures both in distributions and in moments. Computing these performance measures is a different matter,
however. We
denote by $W^j_\infty, W_\infty$ the steady state versions
of the random variables $W^j_k,W_k$. Thus provided that $\rho^*<1$ and some additional
technical assumptions hold, we have
 \begin{align}\label{eq:Winfty}
\lim_{n\rightarrow\infty}\E[W_n]=\E[W_\infty].
\end{align}
We will assume that $\rho^*<1$ holds without explicitly stating it.
Rather than describing  the assumptions required to make (\ref{eq:Winfty}) true, we will simply assume when
stating our results that  (\ref{eq:Winfty}) holds as well.

\subsection{A multiclass single server (MCSS) queueing system. Stochastic model}
We now describe our second queueing model.
Consider a single server  queueing system which
processes $J$ classes of jobs. The jobs of class $j=1,2,\ldots,J$
arrive from outside according to a renewal process with i.i.d.
interarrival times $U^j_k, k\ge 1$ and distribution function
$F_{a,j}(t)=\pr(U^j_1\le t)$. The arrival rate for class $j$ jobs is
$\lambda_j\triangleq 1/\E[U^j_1]$. It is possible that some classes $j$ do not have an external
arrival process, in which case $U^j_k=\infty$ almost surely and $\lambda_j=0$.
Let $\sigma_{a,j}^2$ be the variance of $U^j_1$.
The sequences $(U^j_k, k\ge 1)$ are also
assumed to be independent for different $j$. Let
$\lambda=(\lambda_j)$  denote the J-vector of arrival rates. We let $\lambda_{\max}=\max_{1\leq j\leq J}{\lambda_j}$
and $\lambda_{\min}=\min_{1\leq j\leq J}{\lambda_j}$. We let $A(t)=(A_j(t))$
denote the vector of cumulative number of external
arrivals up to time $t$ where $A_j(t)=\max\{k: \sum_{1\le i\le k}U^j_i\le t\}$.

The jobs corresponding to class $j$ are stored in buffer $B_j$ until served. As in the single class case, we assume all buffers are of infinite capacity. The service time for the $k$-th job arriving to buffer $B_j$ is denoted by $V^j_k$ and the sequence $(V^j_k, k\ge 1)$ is assumed to be i.i.d.
with a common distribution function $F_{s,j}(t)=\pr(V^j_1\le t)$. Additionally, these sequences are assumed to be independent for
all $j$ and independent from the interarrival times sequences  $(U^j_k, k\ge 1)$.
The average service time for class $j$ is $m_j\triangleq E[V^j_1]$ and the service rate is $\mu_j\triangleq 1/\E[V^j_1]$.
$\sigma_{s,j}^2$ denotes the variance of $V^j_1$.
Let $\bar{m} = (m_j)$ denote
the $J$-vector of average service times and let  $\mu=(\mu_j)$ be the $J$-vector of service rates.
Let $M$ denote the diagonal matrix with $j$-th entry equal to $\mu_j$
and let  $\mu_{\max}=\max_{1\leq j\leq J}{\mu_j}$.

We assume that the jobs in buffer $B_j$ are served using FIFO rule, but prioritizing jobs between different buffers $B_j$ is done
using some scheduling policy $\theta$. The only assumption we make about $\theta$ is that it is a
\emph{work-conserving} policy. Namely, the server is working full time as long as there is at least one job in one
of the buffers $B_j, ~1\le j\le J$. The only performance measure we will consider is the workload (defined below)
for which it is well known that
the details of the scheduling policy are unimportant for us, as long as the policy is work-conserving.

The routing of jobs after service completions is determined using a routing
matrix $P$, which is an $J$ by $J$ $0,1$  matrix $P=(P_{i,j}, 1\le i,j\le J)$. It is assumed that
$\sum_j P_{i,j}\le 1$ for each $i$. (Namely, the sum is either $1$ or $0$). Upon service completion in buffer $B_i$,
the job of class $i$ is routed to buffer $j$ if $P_{i,j}=1$. Otherwise, if $\sum_j P_{i,j}=0$, the jobs of class
$i$ leave the network. It is assumed that $P^n=0$ for some positive integer $n$. It is easy to see that this condition
is equivalent to saying that all jobs eventually leave the network.

It is known~\cite{ChenYaoBook} that the traffic equation
$\bar\lambda_i=\lambda_i+\sum_{1\le j\le J} \bar\lambda_jP_{j,i}$ has a unique
solution $\bar\lambda=(\bar\lambda_j)$ given simply as
$\bar\lambda=[I-P^T]^{-1}\lambda$, where $I$ is the $J$ by $J$
identity matrix. Let $\bar\lambda_{\max}=\max_j(\bar\lambda_j)$
(observe that $\lambda_j\le \bar\lambda_j$ for every $j$ and hence $\bar\lambda_{\max}\geq \lambda_{\max}$). Let $\bar A(t)=(\bar A_j(t))$ denote the vector of number of arrivals by time $t$ that will eventually route to server j: $\bar A_j(t) = e_j^T(I+(P^T)^1+(P^T)^2+\ldots)A(t)=e_j^T[I-P^T]^{-1}A(t)$ and $e_j$ denotes the $j-th$ unit
vector.

The traffic intensity vector is defined to be
$\bar\rho=M^{-1}\bar\lambda=M^{-1}[I-P^T]^{-1}\lambda$. The traffic
intensity of the entire server is $\rho=e^T\bar\rho$, where $e$ is
the $J$ vector of ones. Let $Q_j(t)$ denote the queue length in
buffer $j$ at time $t$,  let $Q(t)=(Q_j(t))$. We assume that $Q(0)=1$.
As for the case of TSC model, our results can be extended to the case $Q(0)\geq 0$, but
for the results regarding steady-state behavior, the initialization of queues is irrelevant.
Denote by $W^j_k$ the
waiting time of the $k$-th job arriving into buffer $j$.  We let $W_t$ denote the workload at time $t$. Namely, $W_t$ is the time required to process all the jobs
present in the system at time $t$, in the absence of  the future arrivals.
Note that $W_t$ is also the virtual waiting time at time $t$ when the scheduling policy is FIFO.
Observe that if $t_0$ marks the beginning of a busy period and $t_1$ belongs to the same
busy period (namely, the server was working continuously during the time interval $[t_0,t_1]$), then almost surely
\begin{align}\label{eq:Wt}
W_{t_1}&=\sum_{i=\bar A_1(t_0)}^{\bar A_1(t_1)}V_i^1+\ldots+\sum_{i=\bar A_J(t_0)}^{\bar A_J(t_1)}V_i^J-(t_1-t_0).
\end{align}
The
model described above is denoted by MCSS(St) (Multiclass Single
Server Stochastic) for short. It is known~\cite{dai} that if $\rho<1$,
and some additional technical assumption on interarrival and
service time distributions hold
then  MCSS(St) is stable and enters the steady state in the
same sense as described for the tandem queueing network. While
in this case the steady-state distribution of many performance measures usually depends on the details of
work-conserving policy used, the steady-state distribution of the workload does
not depend on the policy, as we have discussed above. Let $W_\infty$ denote the workload in steady state,
and let $B_\infty$ and $I_\infty$ denote the steady-state duration of the busy and idle periods, respectively.
Additionally, denote by $I_0,B_1,I_1,B_2,I_2,\ldots$ the alternating sequence of the lengths of the busy and idle periods of the
MCSS(St) system, assuming that time zero initiates a busy period.
Under the same technical assumptions as above the following ergodic properties hold almost surely:
\begin{align}
\lim_{t\rightarrow\infty}{\int_0^t W_sds \over t}&=\E[W_\infty], \label{eq:Wtergodic} \\
\lim_{n\rightarrow\infty}{\sum_{1\le i\le n} B_i\over n}&=\E[B_\infty], \label{eq:Bergodic} \\
\lim_{n\rightarrow\infty}{\sum_{1\le i\le n} I_i\over n}&=\E[I_\infty], \label{eq:Iergodic} \\
\lim_{n\rightarrow\infty}{\sum_{1\le i\le n} B^2_i\over n}&=\E[B^2_\infty]. \label{eq:B2ergodic}
\end{align}
We denote by $n(t)$ the number of busy periods that have been initiated up to time $t$. Mathematically, we define $n(t)$ to satisfy $\sum_{1\le i\le n(t)-1}(B_i+I_i)<t\le \sum_{1\le i\le n(t)}(B_i+I_i)$.
When $t\in [\sum_{1\le i\le n(t)-1}(B_i+I_i),\sum_{1\le i\le n(t)-1}(B_i+I_i)+B_{n(t)}]$, $t$ falls on a busy period and using the definition of $n(t)$, we have $W(t)\le B_{n(t)}$. When $t\in [\sum_{1\le i\le n(t)-1}(B_i+I_i)+B_{n(t)},\sum_{1\le i\le n(t)}(B_i+I_i)]$, $t$ falls on idle period $I_{n(t)}$ and hence $W(t)=0$. We let $\tau_i$ denote the beginning of the $i$-th busy period. This implies
\begin{align*}
{\int_0^tW(s)ds\over t} = {\sum_{i=1}^{n(t)}\int_{\tau_i}^{\min(\tau_i+B_i,t)}W(s)ds\over t}
\le {\sum_{1\le i\le n(t)}B_i^2\over \sum_{1\le i\le n(t)-1}(B_i+I_i)}
\end{align*}
If (\ref{eq:Wtergodic}),(\ref{eq:Bergodic}),(\ref{eq:Iergodic}) and (\ref{eq:B2ergodic}) hold, then we also obtain
\begin{align}\label{eq:boundWinfty}
\E[W_\infty]\le {\E[B^2_\infty]\over \E[B_\infty]+\E[I_\infty]}\le {\E[B^2_\infty]\over \E[B_\infty]}.
\end{align}
This bound will turn useful when we apply our results for robust optimization models to the underlying stochastic model.
As for the TSC case, we assume from now on $\rho<1$.
Rather than listing the assumptions leading to  ergodic properties
(\ref{eq:Wtergodic}),(\ref{eq:Bergodic}),(\ref{eq:Iergodic}) and (\ref{eq:B2ergodic}) we assume
when stating our results,
that the stochastic process $W_t$ enters the steady-state
as $t\rightarrow\infty$ and that the properties
(\ref{eq:Wtergodic}),(\ref{eq:Bergodic}),(\ref{eq:Iergodic}) and (\ref{eq:B2ergodic}) holds almost surely.

\subsection{Robust optimization type queueing systems}
We now describe deterministic robust optimization type counterparts of the two
stochastic queueing models described in the previous subsections.

We begin with TSC model and describe the corresponding model which we denote by
TSC(RO) (Tandem Single Class Robust Optimization). The description of the
network topology is the same as for TSC(St). However, it is
not assumed that $U_k,V^j_k$ and, as a result $Q(t),W^j_k,W_k$
are random variables. Rather we assume that these quantities are
\emph{arbitrary} subject to certain linear constraints detailed
below. Additionally, we assume that the system starts empty $Q(0)=0$ and only $n$ jobs
go through the system.

Specifically, consider a sequence  of non-negative deterministic interarrival
and service times $(U_k,1\le k\le n),(V^j_k,1\le k\le n),1\le j\le J$.
Let
\begin{align}\label{eq:phi}
\phi(x)=\left\{
       \begin{array}{ll}
         \sqrt{x\ln\ln x}, & \hbox{$x\ge e^e$;} \\
         1, & \hbox{$x<e^e$.}
       \end{array}
     \right.
\end{align}
We assume that there exist
$\lambda,\Gamma_a$ and $\mu_j,\Gamma_{s,j}\ge 0, 1\le j\le J$ such that
\begin{align}\label{eq:ukLILarrivaltandem}
\big|\sum_{k+1\le i\le n}U_k-\lambda^{-1}(n-k)\big|
&\le
    \Gamma_a\phi(n-k),\qquad k=0,1,\ldots,n-1,\\
\big|\sum_{k+1\le i\le n}V^j_i-\mu_j^{-1}(n-k)\big|
&\le \Gamma_{s,j}\phi(n-k),\qquad k=0,1,\ldots,n-1,~j=1,2,\ldots,J.
\label{eq:ukLILservicetandem}
\end{align}
It is because we need to consider tail summation $\sum_{k+1\le i\le n}$ we assume that only $n$ jobs going through the system,
though we will be able to apply our results in the stochastic setting where infinite number of jobs pass through the system.
Let $\Gamma=\max(\Gamma_a,\Gamma_{s,j})$. Borrowing from the robust optimization literature terminology (\cite{bertsimassim04}), the parameters $\Gamma_a,\Gamma_{s,j},\Gamma$
are called \emph{budgets
of uncertainty}. Note, that the values $U_k, V^j_k, k\ge 1$ uniquely define the corresponding
performance measures
$Q_j(t),W^j_k,W_k, k=1,\ldots,n$. There is no notion of steady state quantities $Q_j(\infty),W_\infty$
for the model TSC(RO).
The motivation for constraints (\ref{eq:ukLILarrivaltandem}) and (\ref{eq:ukLILservicetandem})
comes from the Law of the Iterated Logarithm, and we discuss the connection in
a separate subsection.

We denote the robust optimization counterpart of the MCSS(St) model by MCSS(RO). In this case it turns out to be convenient to consider infinite sequence of jobs. Thus consider infinite sequences of deterministic non-negative values
$(U^j_k,k\ge 1),(V^j_k,k\ge 1),1\le j\le J$. It is assumed that
values $\lambda_j,\mu_j,\Gamma_{a,j},\Gamma_{s,j}\ge 0,~1\le j\le J$ exist such
that
\begin{align}
\big|\sum_{1\le i\le k}U^j_k-\lambda_j^{-1}k\big|&\le \Gamma_{a,j}\phi(k),~~k=1,2,\ldots,
~j=1,2,\ldots,J, \label{eq:ukLILarrivalsserver}\\
\big|\sum_{1\le i\le k}V^j_i-\mu_j^{-1}k\big|&\le \Gamma_{s,j}\phi(k),~~k=1,2,\ldots,
~j=1,2,\ldots,J. \label{eq:ukLILservicesserver}
\end{align}
For convenience we assume that at time zero the system begins with exactly one job in every class $j=1,\ldots,J$:
$Q_j(0)=1$. Then the first after time zero external arrival into buffer $j$ occurs at time $U^j_1$. As before, we let $\Gamma=\max(\Gamma_{a,j},\Gamma_{s,j})$.

For technical reasons, we also assume that $\Gamma$ in TSC(RO), MCSS(RO) constraints satisfies
\begin{align}
\lambda\Gamma\ge e^{2e} \text{~~~~and~~~~} \min_{j}\lambda_j\Gamma\ge e^{2e}, \text{~~~~respectively.} \label{eq:GammaLarge}
\end{align}

\subsection{The Law of the Iterated Logarithm}
One of the cornerstones of the probability theory is the Law of the Iterated Logarithm (LIL) \cite{chungBook}, which states that
given a i.i.d. sequence of random variables $X_1,\ldots,X_n,\ldots$ with zero mean and finite variance
$\sigma$, the following holds almost surely,
\begin{align*}
\limsup_{n\rightarrow\infty}{\sum_{1\le k\le n}X_k\over \sigma\sqrt{2n\ln\ln n}}=1,~~
\liminf_{n\rightarrow\infty}{\sum_{1\le k\le n}X_k\over \sigma\sqrt{2n\ln\ln n}}=-1.
\end{align*}
The LIL extends immediately to non-zero mean i.i.d. sequences by subtracting $n\E[X_1]$ from $\sum_{1\le k\le n}X_k$.
Furthermore, LIL implies (in the case of zero-mean variables) that
\begin{align}\label{eq:LILGamma}
\Gamma_{\text{LIL}}\triangleq
\sup_{n\ge 1}{|\sum_{1\le k\le n}X_k|\over \sigma\sqrt{2}\phi(n)}<\infty,
\end{align}
where $\phi$ is defined in (\ref{eq:phi}).
Note that $\Gamma_{\text{LIL}}$ is a random variable.
Thus when we consider stochastic queueing models such as  TSC(St) or MCSS(St),
the constraints
(\ref{eq:ukLILarrivaltandem}),(\ref{eq:ukLILservicetandem}),(\ref{eq:ukLILarrivalsserver}),(\ref{eq:ukLILservicesserver})
hold with probability one,
with $\Gamma=\sqrt{2}\Gamma_{\text{LIL}}\sigma$, where $\Gamma_{\text{LIL}}$
is defined in (\ref{eq:LILGamma}) for the corresponding random sequence. Specifically,
let  $\Gamma_{a}=\Gamma_{a,\text{LIL}}=\Gamma_{\text{LIL}}$ and $\sigma=\sigma_a$, when
$X_k=U_{n-k}-\lambda^{-1}, 0\le k\le n-1$ and $U_k$ is the sequence of interarrival times in the TSC(St) model.
Similarly define $\Gamma_{s,j}=\Gamma_{s,j,\text{LIL}}$ when $X_k=V^j_{n-k}-\mu_j^{-1}, 0\le k\le n-1, 1\le j\le J$.
Observe, that for $\Gamma_a,\Gamma_{s,j}$ thus defined,
the constraints (\ref{eq:ukLILarrivaltandem}),(\ref{eq:ukLILservicetandem}) hold
for an infinite sequences of jobs (that is jobs which would have indices $-1,-2,\ldots$),
even though we need it only for the first $n$ jobs. For the MCSS(St) model define
$\Gamma_{a,j}=\Gamma_{a,j,\text{LIL}},\Gamma_{s,j}=\Gamma_{s,j,\text{LIL}}$ corresponding to the sequences
$U^j_k-\lambda_j^{-1}, V^j_k-\lambda_j^{-1}, k\ge 1$, respectively.
We obtain
\begin{prop}\label{prop:LIL}
Constraints
(\ref{eq:ukLILarrivaltandem}),(\ref{eq:ukLILservicetandem}),(\ref{eq:ukLILarrivalsserver}),(\ref{eq:ukLILservicesserver})
hold with probability one for
$\Gamma_a=\sqrt{2}\Gamma_{a,\text{LIL}}\sigma_a,$ $\Gamma_{s,j}=\sqrt{2}\Gamma_{s,j,\text{LIL}}\sigma_{s,j},$
$\Gamma_{a,j}=\sqrt{2}\Gamma_{a,j,\text{LIL}}\sigma_{a,j},$ and $\Gamma_{s,j}=\sqrt{2}\Gamma_{s,j,\text{LIL}}\sigma_{s,j}$,
respectively, where $\Gamma_{\cdot,\cdot,\text{LIL}}$ is defined in (\ref{eq:LILGamma}) for the corresponding sequence.
\end{prop}

As a conclusion, for \emph{every} property derivable  on the basis of these constraints
in our deterministic robust optimization  queueing network models, such as,
for example,  bounds on the sojourn time of the $n$-th job in TSC, the \emph{same property} applies
with probability one for the underlying stochastic network. This observation underlies the main idea of the paper.

\section{Main results}\label{section:MainResults}
In this section we state our main results on the performance bounds for robust optimization type queueing networks
TSC(RO) and MCSS(RO), and the implications of our results for their stochastic counterparts  TSC(St) and MCSS(St).
We begin with TSC(RO) with the goal of obtaining a bound on the sojourn time.

\begin{theorem}\label{theorem:MainResultTSCWaitingLog}
The  sojourn time of the $n$-th job  in the TSC(RO) queueing
system with constraints
(\ref{eq:ukLILarrivaltandem}),(\ref{eq:ukLILservicetandem})
satisfies
\bea W_n \le
\frac{7J^2\Gamma^2\lambda}{1-\rho^*}\ln\ln{\frac{J\lambda\Gamma}{1-\rho^*}}+J\lambda^{-1}.
\label{Thmlnln}
\eea
\end{theorem}
Observe that the bound on the sojourn time is explicit. It is expressed directly
in terms of the primitives of the queueing system such as arrival and service rates.
Observe also that the upper bound is independent from $n$. One can think of this bound as a ``steady-state" bound on the
sojourn time in the robust optimization model of the TSC system. Additionally, the constant $\Gamma^2$ is related to the ``variances" of interarrival and service
times viz a vi the LIL (\ref{eq:LILGamma}). It is known that in the stochastic GI/GI/1 queueing system
the expected waiting time in steady state is approximately $(\sigma_a^2+\sigma_s^2)/(2\lambda(1-\rho))$, when
the system is in heavy traffic, namely $\rho\rightarrow 1$. Namely, the expected waiting time depends
linearly on the variances of interarrival and service time. Our bound (\ref{Thmlnln}) is thus
consistent with this type of dependence. On the other hand, unfortunately, our bound depends quadratically
on the number of servers $J$, whereas the correct dependence is known to be linear, at least in some special
cases~\cite{Reiman84},\cite{GamarnikZeevi}.

The bound above does not have a correct $O((1-\rho^*)^{-1})$ scaling, which is known
to be correct from the heavy-traffic theory perspective~\cite{Reiman84},\cite{GamarnikZeevi}.
However, the correction factor is a very slowly
growing function $\ln\ln$. The upshot is that we can use this bound to obtain a bound on $W_n$ and $W_\infty$
in the underlying stochastic system. This is what we do next.

\begin{coro}\label{coro:MainResultTSCWaiting}
For every $n\ge 1$ the sojourn time of the $n$-th job in the TSC(St) queueing network satisfies
\begin{align}\label{eq:TSCWaitingLog}
\E[W_n]\le
\E\Big[\frac{7J^2\Gamma^2\lambda}{1-\rho^*}\ln\ln{\frac{J\lambda\Gamma}{1-\rho^*}}\Big]+J\lambda^{-1}.
\end{align}
where
$\Gamma=\max_j(\sqrt{2}\sigma_a\Gamma_{a,\text{LIL}},\sqrt{2}\sigma_{s,j}\Gamma_{s,j,\text{LIL}},e^{2e}\lambda^{-1})$.
If in addition the assumption (\ref{eq:Winfty}) holds then
\begin{align}\label{eq:TSCWaitinginfty}
\E[W_\infty]\le
\E\Big[\frac{7J^2\Gamma^2\lambda}{1-\rho^*}\ln\ln{\frac{J\lambda\Gamma}{1-\rho^*}}\Big]+J\lambda^{-1}.
\end{align}
\end{coro}
\begin{proof}
We first assume Theorem~\ref{theorem:MainResultTSCWaitingLog} is established. Note, in the context of the stochastic system, both $W_n$ and $\Gamma$ in Theorem~\ref{theorem:MainResultTSCWaitingLog} are random variables. We take $\Gamma=\max_j(\sqrt{2}\sigma_a\Gamma_{a,\text{LIL}},\sqrt{2}\sigma_{s,j}\Gamma_{s,j,\text{LIL}},e^{2e}\lambda^{-1})$ to satisfy (\ref{eq:GammaLarge}), where $\Gamma_{\cdot,\cdot,\text{LIL}}$ is defined in (\ref{eq:LILGamma}) for the corresponding sequence. Applying Proposition~\ref{prop:LIL} we have that (\ref{Thmlnln}) holds with probability one for the underlying stochastic network. The bound (\ref{eq:TSCWaitingLog}) now follows from taking expectations of both sides of (\ref{Thmlnln}).
The bound (\ref{eq:TSCWaitinginfty})
follows from applying (\ref{eq:Winfty}) to (\ref{eq:TSCWaitingLog}).
\end{proof}

\vspace{.2in}
We now turn our attention to the MCSS queueing model. Our approach for deriving a bound on the workload is based on first obtaining an upper bound on the duration of the busy period. Thus, we first give a bound on the duration of the busy period and then turn to the workload. Recall our assumption $Q(0)=1$, though our results can readily be extended to the general case of $Q(0)\geq 0$. Thus, time $t=0$ marks the beginning of a busy period.

\begin{theorem}\label{theorem:MainResultMCSSWaitingLog}
Given a MCSS(RO) queueing system with constraints (\ref{eq:ukLILarrivalsserver}),(\ref{eq:ukLILservicesserver}),
let $B$ be the duration of the busy period initiated at time $0$.
Then
\begin{align}\label{eq:BupperBoundLog}
B\le {5(4J+3)^2\bar\lambda_{\max}^3\Gamma^4\over (1-\rho)^2}\ln\ln{{2(4J+3)\bar\lambda_{\max}^2\Gamma^2\over 1-\rho}},
\end{align}
\vspace{-.5in}
\begin{align}\label{eq:WupperBoundLog}
\text{and~~~}\sup_{0\le t\le B}W(t) \le {2(4J+3)^2\bar\lambda_{\max}^3\Gamma^4\over 1-\rho}
\ln\ln{(4J+3)\bar\lambda_{\max}^2\Gamma^2\over 1-\rho}+\Gamma+3\bar\lambda_{\max}^2\Gamma^3.
\end{align}
\end{theorem}
While the bound (\ref{eq:WupperBoundLog}) corresponds to the maximum workload during a given busy period, the
actual value of the bound does not depend on the busy period length explicitly. As it will become apparent from the
proof, we use the same technique for obtaining a bound simultaneously on the duration  of the busy period and maximum
workload during the busy period. Let us now discuss the implications of these bounds for the underlying stochastic model MCSS(St).

\begin{coro}\label{coro:MainResultMCSSWaiting}
Given a MCSS(St) model, suppose the relations (\ref{eq:Wtergodic}),(\ref{eq:Bergodic}),(\ref{eq:Iergodic}) and (\ref{eq:B2ergodic})
hold. Then
\begin{align}
&\E[B_\infty]\le \E\Big[{5(4J+3)^2\bar\lambda_{\max}^3\Gamma^4\over (1-\rho)^2}
\ln\ln{{2(4J+3)\bar\lambda_{\max}^2\Gamma^2\over 1-\rho}}\Big],\label{eq:MCSSBusy}\\
&\E[W_\infty]\le \E\Big[{25(4J+3)^4\bar\lambda_{\max}^6\mu_{\max}\Gamma^8\over (1-\rho)^4}
\Big(\ln\ln{{2(4J+3)\bar\lambda_{\max}^2\Gamma^2\over 1-\rho}}\Big)^2\Big],\label{eq:MCSSWaiting}
\end{align}
where
$\Gamma=\max_j(\sqrt{2}\sigma_{a,j}\Gamma_{a,j,\text{LIL}},\sqrt{2}\sigma_{s,j}\Gamma_{s,j,\text{LIL}},e^{2e}\lambda_{\min}^{-1})$.
\end{coro}
Unfortunately, in this case the scaling of our bounds as $\rho\rightarrow 1$ deviates significantly from the correct
behavior. From the heavy traffic theory~\cite{DaiKurtz}, the correct behavior for the steady-state workload
should be $O((1-\rho)^{-1})$.  As for the steady-state busy period, the theory of M/G/1 queueing system~\cite{kleinrock}
suggests the behavior $O((1-\rho)^{-{3\over 2}})$ as opposed to $O((1-\rho)^{-2}\ln\ln(1-\rho)^{-1})$ which we obtain. On the positive side,
however, we managed to obtain explicit bounds on the performance measures which are expressed directly in terms of the
stochastic primitives of the model, which we do not believe was possible
using prior methods. We leave it as an interesting open problem to derive the performance bounds
based on the robust optimization technique, which lead to the correct scaling behavior as $\rho\rightarrow 1$.

While the proofs of our main results are technically involved, conceptually they are not complicated.
Before we turn to formal proofs, in order to help the reader, we outline below informally some of the key
proof steps for our results.

For the TSC queueing network we first replace the constraints
(\ref{eq:ukLILarrivaltandem}),(\ref{eq:ukLILservicetandem})
with more general constraints, see
(\ref{eq:Gammaj}) and (\ref{eq:Gamma}) below.
Our results for the TSC network rely mostly on the Lindley's type recursion which in
a single server queueing system recursively represents in the waiting time of the $n$-th job in terms of the interarrival
and service times of the first  $n$ jobs. It is classical result of the queueing theory that this waiting time
can be thought of as maximum of a random walk, with steps equalling in distribution to the difference between the
interarrival and service times. We derive a similar relation in the form of a bound on the sojourn time of the
$n$-th job in the TSC network. This bound is given in Theorem~\ref{theorem:GeneralUncertainty}.
Then we view this bound as an optimization problem and obtain a bound on the objective value by proving the concavity
of the objective function and substituting explicit bounds from constraints
(\ref{eq:ukLILarrivaltandem}),(\ref{eq:ukLILservicetandem}).

Our proofs for the MCSS queueing system rely on the relation (\ref{eq:Wt}). Namely, we take advantage of the fact
that the workload is depleted with the unit rate during the busy period. Then we take advantage of the constraints
(\ref{eq:ukLILarrivalsserver}),(\ref{eq:ukLILservicesserver}) to show that in the MCSS(RO) system the workload at time $t$ during the busy period can be upper bounded by an expression of the form $-at+b\sqrt{t\ln\ln t}+c$ with strictly positive $a,b$.
It is then not hard to obtain an explicit estimated $t_0$ such that this expression is negative for $t>t_0$. Since this expression
is an upper bound on a non-negative quantity (workload), then the duration of the busy period cannot be larger than $t_0$.
This leads to an upper bound on the duration of the busy period in the MCSS(RO) system. In order to obtain a bound on the workload, we again
take advantage of (\ref{eq:Wt}) and further obtain explicit upper bounds on the terms involving the sums of service times. We show that the workload at time $t$ is at most $-at+b\sqrt{t\ln\ln t}+c$. We then obtain an upper bound on the
workload during the busy period by obtaining explicit bounds on $\max_{t\ge 0} -at+b\sqrt{t\ln\ln t}+c$.

Our derivation of the bounds for the stochastic model MCSS(St) relies on the ergodic representation
(\ref{eq:Wtergodic}). We consider a modified system in which each busy period is initiated with simultaneous
arrival of one job into \emph{every} buffer $j$. This leads to a alternating renewal process with alternating
i.i.d. busy and idle periods. We then obtain a bound on the steady-state workload in terms of the second
moment of the busy period in the modified queueing system, using the renewal theory type arguments. It is
this necessity to look at the second moment of the busy period which leads to a conservative scaling
$O\Big((1-\rho)^{-4}(\ln\ln(1-\rho)^{-1})^2\Big)$ in our bound (\ref{eq:MCSSWaiting}) on the steady-state
workload.

\section{Tandem single class queueing system analysis: proof of Theorem \ref{theorem:MainResultTSCWaitingLog}}\label{section:tandemAnalysis}
In order to prove Theorem~\ref{theorem:MainResultTSCWaitingLog} we first generalize constraints (\ref{eq:ukLILarrivaltandem}),(\ref{eq:ukLILservicetandem}) and obtain a method for bounding $W_n$ under more general uncertainty assumptions.

\subsection{General upper bound on the sojourn times}
Given a sequence of non-negative real values
$\Gamma_{\min}^{j}(k), \Gamma_{\max}^{j}(k)$ $1\leq j\leq J,1\leq k\leq n$, $\Gamma_{\min}(k), \Gamma_{\max}(k)$ $1\leq k \leq n$,
we consider the set of
all sequences of service times and interarrival times $(V_i^j),(U_i)$ $j=1,\ldots,J$, $i=1,\ldots,n$ satisfying
for all $k=1,\ldots,n$
\begin{align}
\Gamma_{\min}^{j}(k)&\leq \sum_{i=k}^{n}V_{i}^j \leq \Gamma_{\max}^{j}(k), \label{eq:Gammaj} \\
\Gamma_{\min}(k)&\leq \sum_{i=k}^{n}U_{i} \leq \Gamma_{\max}(k), \label{eq:Gamma}\\
V_i^j, U_i &\geq 0. \notag
\end{align}

In the next theorem we obtain a bound on the sojourn time of the $n$-th job in TSC(RO) system in terms of values $\Gamma_{\min}^{j}(k),\Gamma_{\max}^{j}(k),\Gamma_{\min}(k),\Gamma_{\max}(k)$.

\begin{theorem}\label{theorem:GeneralUncertainty}
Suppose the relations (\ref{eq:Gammaj}) and (\ref{eq:Gamma}) hold. Then
\bea W_n \leq \max_{n\geq k_J\geq\ldots\geq k_1\geq
1}\sum_{j=1}^{J-1}\big(\Gamma^{j}_{\max}(k_j)-\Gamma^{j}_{\min}(k_{j+1}+1)\big)+\Gamma^{J}_{\max}(k_J)-\Gamma_{\min}(k_1+1)
\eea
\end{theorem}

We now show how Theorem~\ref{theorem:GeneralUncertainty} implies our main result Theorem~\ref{theorem:MainResultTSCWaitingLog}.

\begin{proof}[Proof of Theorem \ref{theorem:MainResultTSCWaitingLog}]
The proof consists of two steps: the first step uses Theorem \ref{theorem:GeneralUncertainty} to
bound $W_n$ with uncertainty sets
(\ref{eq:ukLILarrivaltandem}),(\ref{eq:ukLILservicetandem}). The second step involves solving
some associated maximization problem.

We set $\Gamma_{\min}(k)=\lambda^{-1}(n+1-k)-\Gamma_a\phi({n+1-k}),
\Gamma_{\max}(k)=\lambda^{-1}(n+1-k)+\Gamma_a\phi({n+1-k}),
\Gamma^j_{\min}(k)=\mu_j^{-1}(n+1-k)-\Gamma_{s,j}\phi({n+1-k}),
\Gamma^j_{\max}(k)=\mu_j^{-1}(n+1-k)+\Gamma_{s,j}\phi({n+1-k})$, where $\phi$ is defined by (\ref{eq:phi}).
From Theorem \ref{theorem:GeneralUncertainty}
we obtain:
\begin{align*}
 W_n &\leq \underset{n\geq k_J \geq \ldots\geq k_1\geq 1}{\max}\sum_{j=1}^{J-1}\big(\mu_j^{-1}(n+1-k_j)+\Gamma_{s,j}\phi{(n+1-k_{j})}\big)\\
 &~~~-\sum_{j=1}^{J-1}\big(\mu_j^{-1}(n+1-k_{j+1}-1)-\Gamma_{s,j}\phi{(n+1-k_{j+1}-1)}\big)\\
&~~~+\big(\mu_J^{-1}(n+1-k_J) + \Gamma_{s,J}\phi{(n+1-k_J)}\big)~-~\big(\lambda^{-1}(n+1-k_1-1)-\Gamma_a \phi{(n+1-k_1-1)}\big)
\end{align*}
Since $n\geq k_{j+1}\ge k_j ~~\forall j$, we can replace $\mu_j^{-1}$ by
$\mu_{\min}^{-1}=\max(\mu_1^{-1},\mu_2^{-1}\ldots,\mu_J^{-1})<\lambda^{-1}$ and preserve
inequality. Similarly, we can replace
$\Gamma_{s,1},\Gamma_{s,2},\ldots,\Gamma_{s,J},\Gamma_a$ by $\Gamma$. We obtain:
\begin{align*}
 W_n &\leq \underset{n\geq k_J \geq \ldots\geq k_1\geq 1}{\max}\sum_{j=1}^{J-1}\Big[\mu_{\min}^{-1}\big(k_{j+1}+1-k_j\big)+
\Gamma\big(\phi{(n+1-k_j)}+ \phi{(n-k_{j+1})}\big)\Big]\\
&~~~+\big(\mu_{\min}^{-1}(n+1-k_J) + \Gamma\phi{(n+1-k_J)}\big)-\big(\lambda^{-1}(n-k_1)-\Gamma \phi{(n-k_1)}\big)\\
&\leq \underset{n\geq k_1\geq 1}{\max}\mu_{\min}^{-1}(n-k_1)+
2J\Gamma\phi{(n+1-k_1)}\\
&~~~+J\mu_{\min}^{-1} -\lambda^{-1}(n-k_1)     \text{~~ where we used $k_1\leq k_2\leq\ldots\leq k_J$ to combine $\Gamma$ terms}\\
&= \underset{n\geq k_1 \geq 1}{\max}(n+1-k_1)(\mu_{\min}^{-1}-\lambda^{-1})+
2J\Gamma\phi{(n+1-k_1)}+(J-1)\mu_{\min}^{-1}+\lambda^{-1}\\
&\leq \underset{n\geq k_1 \geq 1}{\max}(n+1-k_1)(\mu_{\min}^{-1}-\lambda^{-1})+
2J\Gamma\phi{(n+1-k_1)}+J\lambda^{-1}    \text{~~~since $\lambda^{-1}>\mu_{\min}^{-1}$}
\end{align*}
We let $x=n+1-k_1$. Since $1\leq k_1 \leq n$ we have that $1\leq x\leq n$ and obtain:
\begin{align}
W_n&\leq \underset{n\geq x \geq 1}{\max}x(\mu_{\min}^{-1}-\lambda^{-1})+
2J\Gamma\phi{(x)}+J\lambda^{-1} \notag\\
&\leq \underset{x \geq 1}{\max}~x(\mu_{\min}^{-1}-\lambda^{-1})+
2J\Gamma\phi{(x)}+J\lambda^{-1}\label{equation:TSC(LOG)_Obj}
\end{align}
Putting $a=\lambda^{-1}-\mu_{\min}^{-1}, b=J\Gamma, c=J\lambda^{-1}$, and using the assumption
(\ref{eq:GammaLarge}), we have
$b/a=\lambda J\Gamma/(1-\rho^*)\ge e^{2e}$, namely, the condition (\ref{eq:bovera}) is satisfied.
Applying Proposition~\ref{prop:Umax} from Appendix  we obtain
\begin{align*}
W_n\le {7\lambda J^2\Gamma^2\over 1-\rho^*}\ln\ln{\lambda J\Gamma\over 1-\rho^*}+J\lambda^{-1}.
\end{align*}
This completes the proof of the theorem.
\end{proof}

\subsection{Proof of Theorem \ref{theorem:GeneralUncertainty}}
Job $1$ enters the system first,
followed by jobs $2,3,\ldots,n$.
Let $U^j_i$ be the time between the arrival of job $i$ and job $i-1$ into server $j$ for $i=2,\ldots,n$ and $j=1,\ldots,J$. Specifically, $U^1_i=U_i$, and we define $U^j_1=V^{j-1}_1$ for $j=2,\ldots,J$. The following relations are well known in the queueing theory~\cite{kleinrock}.
\bea
W_i^j=\max(W_{i-1}^{j}+V_{i-1}^{j}-U_i^{j},0)
&&\text{~~~ $\forall$ $i=2,\ldots,n$, $j=1,\ldots,J$,}
\label{basicrelation} \\
U_i^j=V_i^{j-1}+I_i^{j-1}
&&\text{~~~ $\forall$ $i=2,\ldots,n$, $j=2,\ldots,J$,}
\label{lemma:L1}\\
W_i^j=\max\Big\{\max_{1\leq k\leq
i-1}\sum_{l=k}^{i-1}\big(V_{l}^j-U_{l+1}^j\big),0\Big\}
&&\text{~~~ $\forall$ $i=2,\ldots,n$, $j=1,\ldots,J$,}
\label{lemma:L2}\\
W_{i-1}^j=W_i^j-I_i^j-(V_{i-1}^j-U_i^j)
&&\text{~~~ $\forall$ $i=2,\ldots,n$, $j=1,\ldots,J$.}
\label{lemma:L3}
\eea

 We now prove some more detailed
results regarding the dynamics of our queueing system.

\begin{coro}\label{coro:C1}
The following relations  hold for  $k=2,\ldots,n-1$: \bea
\sum_{i=k+1}^{n}U_i^2 = \sum_{i=k+1}^{n}(V_{i}^1+I_{i}^1)
=W_n^1-W_{k}^1+\sum_{i=k+1}^{n}U_i^1+V_n^1-V_{k}^1 \nonumber.
\eea
\end{coro}
\begin{proof} The first equality follows from (\ref{lemma:L1}). To prove the second equality
we use (\ref{lemma:L3}) to obtain
\begin{align*}
\sum_{i=k+1}^{n}(V_{i}^1+I_{i}^1)&= \sum_{i=k+1}^{n} (W_{i}^1-W_{i-1}^1+U_i^1)+V^1_n-V^1_k,
\end{align*}
and the result follows.
\end{proof}

\begin{lemma}\label{lemma:L4}
\bea
W_n=\underset{n\geq k_J\geq\ldots\geq k_1\geq1}{\max}\sum_{i=k_1}^{k_2}V_i^1
+\sum_{i=k_2}^{k_3}V_i^2+\ldots+\sum_{i=k_{J-1}}^{k_{J}}V_i^{J-1}+\sum_{i=k_J}^{n}V_i^J-\sum_{i=k_1+1}^{n}U_i^1.
\eea
\end{lemma}
\begin{proof}
We prove Lemma \ref{lemma:L4} by induction. We let $W_i^{j,S}=W_i^j+V_i^j$ denote the sojourn time of customer $i$ in server $j$.\\
\emph{Case $J=1$:} We first define $\sum_{i=j+1}^j\equiv0$ for all $j$. Using
(\ref{lemma:L2}) and $V_i^j\geq 0$ we have for any $n=2,\ldots,n$:
\beaa W_n^{1,S}&= &\max\Big(\underset{n-1\geq k_1\geq 1}{\max}\sum_{i=k}^{n-1}\big(V_{i}^1-U_{i+1}^1\big),0\Big) +V_{n}^1\\
&=&\max\Big(\underset{n\geq k_1\geq 1}{\max}\sum_{i=k_1}^{n}V_{i}^1-\sum_{i=k_1+1}^{n}U_i^1,V_n^1\Big) \\
&=&\underset{n\geq k_1\geq 1}{\max}\Big(\sum_{i=k_1}^{n}V_{i}^1-\sum_{i=k_1+1}^{n}U_i^1\Big)
\text{~~~~and this completes case $J=1.$}
\eeaa
\emph{Case $J>1$:}
Note that $W_n=W_n^{1,S}+(W_n^{2,S}+\ldots+W_n^{J,S})$ and denotes the sojourn time of job $n$ in $J$-server system. We suppose that the result holds for a $J-1$ tandem system and proceed by induction:
\beaa
 &&\underset{n\ge  k_J\geq\ldots\geq k_1\geq1}{\max}\sum_{i=k_1}^{k_2}V_i^1+\sum_{i=k_2}^{k_3}V_i^2+\ldots+\sum_{i=k_{J-1}}^{k_{J}}V_i^{J-1}+\sum_{i=k_J}^nV_i^J-\sum_{i=k_1+1}^nU_i^1\\
 &=&\underset{n\ge  k_J\geq\ldots\geq k_1\geq1}{\max}\Big(\sum_{i=k_1}^{k_2}V_i^1-\sum_{i=k_1+1}^{k_2}U_i^1\Big)-\sum_{i=k_2+1}^nU_i^1+\sum_{i=k_2}^{k_3}V_i^2+\ldots+\sum_{i=k_{J-1}}^{k_{J}}V_i^{J-1}+\sum_{i=k_J}^nV_i^J\\
 &=&\underset{n\ge  k_J\geq\ldots\geq k_2\geq1}{\max}\Big[\underset{k_1: k_2\geq k_1\geq 1}{\max}\Big(\sum_{i=k_1}^{k_2}V_i^1-\sum_{i=k_1+1}^{k_2}U_i^1\Big)-\sum_{i=k_2+1}^nU_i^1+\sum_{i=k_2}^{k_3}V_i^2+\ldots+\sum_{i=k_{J-1}}^{k_{J}}V_i^{J-1}+\sum_{i=k_J}^nV_i^J\Big]\\
 &=&\underset{n\ge  k_J\geq\ldots\geq k_2\geq1}{\max}W_{k_2}^{1,S}-\sum_{i=k_2+1}^nU_i^1+\sum_{i=k_2}^{k_3}V_i^2+\ldots+\sum_{i=k_{J-1}}^{k_{J}}V_i^{J-1}+\sum_{i=k_J}^nV_i^J
  \text{~~the base case $J=1$ is used}\\
 &=&\underset{n\ge  k_J\geq\ldots\geq k_2\geq1}{\max}\Big(W_{k_2}^{1,S}-\sum_{i=k_2+1}^nU_i^1\Big)+\sum_{i=k_2}^{k_3}V_i^2+\ldots+\sum_{i=k_{J-1}}^{k_{J}}V_i^{J-1}+\sum_{i=k_J}^nV_i^J\\
 &=&\underset{n\ge  k_2\geq\ldots\geq k_J\geq1}{\max}\Big(W_n^{1,S}-\sum_{i=k_2+1}^nU_i^2\Big)+\sum_{i=k_2}^{k_3}V_i^2+\ldots+\sum_{i=k_{J-1}}^{k_{J}}V_i^{J-1}+\sum_{i=k_J}^nV_i^J\\
 &&\text{we used Corollary \ref{coro:C1} and  $W_{k_2}^{1,S}=W_{k_2}^1+V_{k_2}^1$}\\
 &=&W_n^{1,S}+\underset{n\ge  k_J\geq\ldots\geq k_2\geq1}{\max}\sum_{i=k_2}^{k_3}V_i^2+\ldots+\sum_{i=k_{J-1}}^{k_{J}}V_i^{J-1}+\sum_{i=k_J}^nV_i^J-\sum_{i=k_2+1}^nU_i^2\\
 &=&W_1^{1,S}+(W_n^{2,S}+\ldots+W_n^{J,S})
 \text{~~by inductive assumption on $J-1$ server system}
\eeaa
and the proof follows from definition of sojourn time $W_n$.
\end{proof}

\begin{proof}[Proof of Theorem \ref{theorem:GeneralUncertainty}]
The result follows immediately from Lemma~\ref{lemma:L4}.
\end{proof}

\section{Multiclass single server analysis: proofs of main results}\label{section:MCSS_Analysis}

\subsection{Proof of Theorem \ref{theorem:MainResultMCSSWaitingLog}} \label{subsection:MCSS(RO)sqrtlnln}

\begin{lemma} \label{lemma:1subsection:MCSS(RO)sqrtlnln}
For every $t$ satisfying
\begin{align}
t\ge\max_j(\lambda_j^{-1}e^e,\lambda_j^{-1}+3\lambda_j^{-1}\lambda_{\max}^2\Gamma^2),\label{eq:lowert}
\end{align}
the following holds: $A_j(t)\leq t\lambda_j+3\lambda_j^2\Gamma^2\phi(t\lambda_j).$
\end{lemma}

\begin{proof}
Assume first $A_j(t)<e^e$. Then applying (\ref{eq:ukLILarrivalsserver}) corresponding to the case $A_j(t)<e^e$,
we obtain $A_j(t)\lambda_j^{-1}-\Gamma_{a,j}\leq t$, namely $A_j(t)\le \lambda_j t+\lambda_j\Gamma_{a,j}\le \lambda_j t+\lambda_j\Gamma$.
Since  $\lambda_j\Gamma,\phi(t\lambda_j) \geq 1$ from (\ref{eq:GammaLarge}) and (\ref{eq:phi}), the desired result is obtained.
For the rest of the proof assume $A_j(t)\geq e^e$.
Applying (\ref{eq:ukLILarrivalsserver}), we obtain $A_j(t)\lambda_j^{-1}-\Gamma_{a,j}\sqrt{A_j(t)\ln\ln A_j(t)}\leq t$.
Which gives
\bea
\frac{A_j(t)-t\lambda_j}{\sqrt{A_j(t)\ln\ln A_j(t)}}&\leq \lambda_j\Gamma_{a,j} &\leq \lambda_j\Gamma\label{k_bound}.
\eea
\vspace{-.1in}
Define $b_j$ by: $b_j = t\lambda_j+3\lambda_j^2\Gamma^2\sqrt{t\lambda_j\ln\ln{t\lambda_j}}$.
Observe that: \beaa
\frac{b_j-t\lambda_j}{\sqrt{b_j\ln\ln b_j}}
&=&\frac{3\lambda_j^2\Gamma^2\sqrt{t\lambda_j\ln\ln{t\lambda_j}}}{\Big((t\lambda_j
+3\lambda_j^2\Gamma^2\sqrt{t\lambda_j\ln\ln{t\lambda_j}})\ln\ln({t\lambda_j
+3\lambda_j^2\Gamma^2\sqrt{t\lambda_j\ln\ln{t\lambda_j}}})\Big)^{\frac{1}{2}}}\\
&\geq&\frac{3\lambda_j^2\Gamma^2\sqrt{t\lambda_j\ln\ln{t\lambda_j}}}
{\Big((t\lambda_j+3\lambda_j^2\Gamma^2\sqrt{t^2\lambda_j^2})
\ln\ln({t\lambda_j+3\lambda_j^2\Gamma^2\sqrt{t^2\lambda_j^2}})\Big)^{\frac{1}{2}}}
\text{~~~since $t\lambda_j\geq\ln\ln{t\lambda_j}$ for $t\lambda_j\geq e^e$ from (\ref{eq:lowert})} \\
&=&\frac{3\lambda_j^2\Gamma^2\sqrt{t\lambda_j\ln\ln{t\lambda_j}}}
{\Big((t\lambda_j)(1+3\lambda_j^2\Gamma^2)\ln\ln{(t\lambda_j)}(1+3\lambda_j^2\Gamma^2)\Big)^{\frac{1}{2}}}\\
&\geq&\frac{3\lambda_j^2\Gamma^2\sqrt{t\lambda_j\ln\ln{t\lambda_j}}}
{\Big((t\lambda_j)(1+3\lambda_j^2\Gamma^2)\ln\ln(t\lambda_j)^2\Big)^{\frac{1}{2}}}
~~\text{since $t\lambda_j>1+3\lambda_j^2\Gamma^2$ from (\ref{eq:lowert})}\\
&\geq&\frac{3\lambda_j^2\Gamma^2\sqrt{\ln\ln{t\lambda_j}}}{\sqrt{(4\lambda_j^2\Gamma^2)(2\ln\ln{t\lambda_j})}}
\text{~~~since $2\ln\ln t\lambda_j>\ln\ln(t\lambda_j)^2$ for $t\lambda_j\geq e^e$ and  $\lambda_j\Gamma\geq1$}\\
&\geq&\lambda_j\Gamma \text{~~~ by simplifying above expression.}
\eeaa
Since $\frac{x-t\lambda_j}{\sqrt{x\ln\ln x}}$ is an increasing function for $x\geq e^e$
and from (\ref{k_bound}), we have that $b_j\geq A_j(t)$ and the result is obtained.
\end{proof}

We now obtain an upper bound on the cumulative arrival processes  $\bar A_j(t), 1\le j\le J$.
\begin{lemma}\label{lemma:2subsection:MCSS(RO)sqrtlnln}
For every $t$ satisfying (\ref{eq:lowert}), the following holds
\begin{align*}
\phi(\bar A_j(t)) &\le \big((2+6\lambda_{\max}^2\Gamma^2)\big)^{1\over 2}\phi(\bar\lambda_j t)
\end{align*}
\end{lemma}

\begin{proof}
Consider first the case $\bar A_j(t)<e^e$. From (\ref{eq:phi}), we have that $\phi(\bar A_j(t))=1$ and applying (\ref{eq:lowert}), the lemma follows. Now we consider the case $\bar A_j(t)\geq e^e$. Recall that $\bar A_j(t)=e_j^T[I-P^T]^{-1}A(t)$. Applying Lemma~\ref{lemma:1subsection:MCSS(RO)sqrtlnln}
\begin{align*}
\bar A_j(t)&\le e_j^T[I-P^T]^{-1}\lambda t+e_j^T[I-P^T]^{-1}\left[
    \begin{array}{c}
        3\lambda_1^2\Gamma^2\phi(t\lambda_1) \\
        3\lambda_2^2\Gamma^2\phi(t\lambda_2) \\
        \vdots \\
        3\lambda_J^2\Gamma^2\phi(t\lambda_J) \\
    \end{array}
        \right] \\
&\le e_j^T[I-P^T]^{-1}\lambda t+3\lambda_{\max}^2\Gamma^2e_j^T[I-P^T]^{-1}\lambda t\qquad \text{applying (\ref{eq:lowert})
and $x\ge \phi(x)$ for  $x\ge e^e$} \\
&=\bar\lambda_j t(1+3\lambda_{\max}^2\Gamma^2),
\qquad \text{applying the definition of $\bar\lambda_j$}.
\end{align*}
Applying this bound we also obtain
\begin{align*}
\ln\ln\bar A_j(t)&\le \ln\ln(\bar\lambda_j t(1+3\lambda_{\max}^2\Gamma^2)) \\
&\le \ln\ln (\bar \lambda_j t)^2 \qquad \text{using assumption (\ref{eq:lowert})} \\
&=\ln\ln\bar \lambda_j t+\ln 2 \\
&\le 2\ln\ln\bar\lambda_jt, \qquad \text{using $\bar\lambda_jt\ge \lambda_j t\ge e^e$ from (\ref{eq:lowert})}.
\end{align*} Combining the previous bounds with definition of $\phi(x)$, the lemma follows.
\end{proof}

\begin{lemma}\label{lemma:3subsection:MCSS(RO)sqrtlnln}
For every $t$ satisfying (\ref{eq:lowert}), we have: $\bar m^T \bar A(t)-t\le (\rho-1)t+3\lambda_{\max}\Gamma^2\phi(\lambda_{\max}t)$.
\end{lemma}
\begin{proof}
Applying definition of $\bar A_j(t)$, we have
\begin{align*}
\bar m^T \bar A(t)-t&= \bar m^T[I-P^T]^{-1}A(t) - t \\
&\le m^T[I-P^T]^{-1}\big(\lambda t + 3\lambda_{\max}\Gamma^2\phi(\lambda_{\max} t)\lambda\big)-t \qquad \text{~from Lemma~\ref{lemma:1subsection:MCSS(RO)sqrtlnln}} \\
&= \sum_j m_j\bar\lambda_j t+3\lambda_{\max}\Gamma^2\phi(\lambda_{\max} t)\sum_j m_j\bar \lambda_j-t \qquad \text{applying the definition of $\bar\lambda_j$}\\
&=(\rho-1)t+3\lambda_{\max}\Gamma^2\phi(\lambda_{\max} t)\rho
\end{align*}
and the lemma follows from applying the condition $\rho<1$ to the second term.
\end{proof}

We now obtain an upper bound in the duration of the busy period. Recall the identity (\ref{eq:Wt}).
Since the busy period begins at time zero its duration is upper bounded by the first time $t$ such that
\begin{align}
\sum_{i=1}^{\bar A_1(t)}V_i^1+\ldots+\sum_{i=1}^{\bar A_J(t)}V_i^J-t<0. \label{eq:workloadnegative}
\end{align}
Consider any $t$ satisfying the lower bound (\ref{eq:lowert}). We have
\begin{align*}
\sum_{i=1}^{\bar A_1(t)}V_i^1&+\ldots+\sum_{i=1}^{\bar A_J(t)}V_i^J-t \\
&\leq \sum_{j=1}^{J}\mu_j^{-1}\bar A_j(t)+\sum_{j=1}^{J}\Gamma_{a,j}\phi(\bar A_j(t))-t
\text{~~applying (\ref{eq:ukLILarrivalsserver}),(\ref{eq:ukLILservicesserver})}\\
&\leq \bar m^T\bar A(t)-t+\sum_{j=1}^{J}\Gamma_{a,j}\big((2+6\lambda_{\max}^2\Gamma^2)\big)^{1\over 2}\phi(\bar\lambda_j t)\text{~~ applying Lemma \ref{lemma:2subsection:MCSS(RO)sqrtlnln}}\\
&\leq t(\rho-1)+3\lambda_{\max}\Gamma^2\phi(\lambda_{\max} t)
+\sum_{j=1}^J\Gamma(2+6\lambda_{\max}^2\Gamma^2)^{1\over 2}\phi(\bar\lambda_j t) \text{~~ applying Lemma \ref{lemma:3subsection:MCSS(RO)sqrtlnln}}\\
&\le t(\rho-1)
+(4J+3)\bar\lambda_{\max}\Gamma^2\phi(\bar\lambda_{\max} t),
\end{align*}
where we have used a crude estimate $2+6\lambda_{\max}^2\Gamma^2<16\lambda_{\max}^2\Gamma^2$, justified by (\ref{eq:GammaLarge}).
We now apply Lemma~\ref{lemma:Unegative} with
$x=\bar\lambda_{\max}t, a=\bar\lambda_{\max}^{-1}(1-\rho),b=(4J+3)\bar\lambda_{\max}\Gamma^2/2$ and $c=0$.
The condition (\ref{eq:bovera}) is implied by assumption (\ref{eq:GammaLarge}), and the second condition of Lemma~\ref{lemma:Unegative} is satisfied since $c=0$.
We obtain that (\ref{eq:workloadnegative}) holds for all $t$ satisfying (\ref{eq:lowert}) and
\begin{align*}
t&\ge {18(4J+3)^2\bar\lambda_{\max}^2\Gamma^4\over 4\bar\lambda_{\max}\bar\lambda_{\max}^{-2}(1-\rho)^2}
\ln\ln{{3(4J+3)\bar\lambda_{\max}\Gamma^2\over 2\bar\lambda_{\max}^{-1}(1-\rho)}} \\
&\ge{5(4J+3)^2\bar\lambda_{\max}^3\Gamma^4\over (1-\rho)^2}
\ln\ln{{2(4J+3)\bar\lambda_{\max}^2\Gamma^2\over 1-\rho}}.
\end{align*}
Observe using  (\ref{eq:GammaLarge})
that the right-hand side of the last expression is larger than the right-hand side of (\ref{eq:lowert}).
Combining two cases we obtain (\ref{eq:BupperBoundLog}).

We now turn to (\ref{eq:WupperBoundLog}).
First suppose $t$ does not satisfy (\ref{eq:lowert}).
Denote
the right-hand side of (\ref{eq:lowert}) by $C$. That is $t<C$.
Observe that $W(t)\le (C-t)+W(C)\le C+W(C)$ as the workload at time $C$ corresponds in addition to arrivals during
$[t,C]$. So now we focus on the case when $t$ satisfies (\ref{eq:lowert}).
We use Proposition~\ref{prop:Umax} from Appendix  and obtain
\begin{align*}
\sup_{C\le t\le B}W(t)&\le {7(4J+3)^2\bar\lambda_{\max}^2\Gamma^4\over 4\bar\lambda_{\max}^{-1}(1-\rho)}
\ln\ln{(4J+3)\bar\lambda_{\max}\Gamma^2\over 2\bar\lambda_{\max}^{-1}(1-\rho)}\\
& \le {2(4J+3)^2\bar\lambda_{\max}^3\Gamma^4\over 1-\rho}
\ln\ln{(4J+3)\bar\lambda_{\max}^2\Gamma^2\over 1-\rho}.
\end{align*}
From (\ref{eq:GammaLarge}), we have $\Gamma\geq \lambda_{\min}^{-1}$. We conclude that
\begin{align*}
\sup_{0\le t\le B}W(t)&\le
{2(4J+3)^2\bar\lambda_{\max}^3\Gamma^4\over 1-\rho}
\ln\ln{(4J+3)\bar\lambda_{\max}^2\Gamma^2\over 1-\rho}+\Gamma+3\bar\lambda_{\max}^2\Gamma^3.
\end{align*}
This completes the proof of the theorem.

\subsection{Proof of Corollary~\ref{coro:MainResultMCSSWaiting}}
First we establish bound (\ref{eq:MCSSBusy}). Let $t=0$ mark the beginning of a busy period with (random)
length $B_\infty$ in steady state.
This means that there is an arrival into one of the classes $j_0$ at time $0$. Consider a modified system
where the first arrivals into classes
$j\ne j_0, \lambda_j>0$ after time $0$ are artificially pushed down to exactly time $0$. Namely, now
at time zero there is an arrival
into every class $j$ with $\lambda_j>0$. The subsequent arrivals into these classes are also pushed earlier by the same
amount, thus creating an i.i.d. renewal process initiated at time $0$. Let $\hat B$ be the busy period initiated
in the modified system at time $0$. It is easy to see that almost surely
$\hat B\ge B_\infty$. However, now that we have arrivals in every class at time zero, applying Proposition~\ref{prop:LIL}
and our result for the robust optimization counterpart queueing system, namely applying part
(\ref{eq:BupperBoundLog}) of Theorem~\ref{theorem:MainResultMCSSWaitingLog}, we obtain the required bound
by taking the expected values of both sides of (\ref{eq:BupperBoundLog}). This establishes part (\ref{eq:MCSSBusy}).

In order to prove (\ref{eq:MCSSWaiting}), we use a bound (\ref{eq:boundWinfty}).
Using our earlier argument for the proof of (\ref{eq:MCSSBusy}) but applying it to the second moment of $\hat B$
we obtain
\begin{align*}
\E[B_\infty^2]\le \E[\hat B^2]\le
\E\Big[{25(4J+3)^4\bar\lambda_{\max}^6\Gamma^8\over (1-\rho)^4}
\Big(\ln\ln{{2(4J+3)\bar\lambda_{\max}^2\Gamma^2\over 1-\rho}}\Big)^2\Big].
\end{align*}
On the other hand, we trivially have
have $\E[B_\infty]\ge \min_{1\le j\le J}m_j=1/\mu_{\max}$, since every busy period involves at least one service completion.
The result then follows.

\section{Conclusion}\label{section:Conclusion}
Using ideas from  the robust optimization theory we have developed a new method
for conducting performance analysis of queueing networks. The essence of our approach is replacing stochastic primitives of the underlying queueing system with deterministic quantities which satisfy the implications of some probability laws. These implications take the form of linear constraints and for the case of two queueing systems, namely Tandem Single Class queueing networks and Multiclass Single Server queueing system, we have managed to derive explicit upper bounds on some performance measures such as sojourn times and workloads. Then we showed that the bounds implied by the Law of the Iterated Logarithm are applicable for the underlying stochastic queueing system leading to explicit and non-asymptotic performance bounds on the same performance measures. We are not aware of any other method of performance analysis which can provide similar performance bounds in queueing model of similar generality.

We have just scratched the surface of possibilities in this paper and we certainly expect that our approach
can be strengthened and  extended in multiple directions, some of which we outline below.
First we expect that our approach extends to
even more general models, such as, for example multiclass queueing networks or  more general
processing networks~\cite{Harrison2002}. The performance bounds can be obtained perhaps again by
introducing linear constraints implied by probability laws and using some sort of a Lyapunov function
for obtaining bounds in the resulting robust optimization type queueing model. Another important direction is identifying new probability laws which lead to tighter constraints than the ones implied by the LIL.
Ideally, one would like to be able to obtain bounds which faithfully represent the scaling behavior of the
performance measures of interest in the heavy traffic regime as the (bottleneck) traffic intensity $\rho$ converges
to the unity. Further, it would be interesting to obtain performance bounds on the tail probability of the
performance measure of interest, perhaps by constructing constraints implied by bounds on the tail probabilities
of the underlying stochastic processes. For example, perhaps one can obtain large deviations type bounds
by considering the linear constraints implied by the large deviations bounds on the underlying stochastic
processes. Deeper connection between the results of this paper and the results in the adversarial queueing theory and the related queueing
literature is worth investigating as well.

Finally, we expect that the philosophy of replacing the \emph{probability model} with \emph{implications
of the probability model} will prove useful in non-queueing contexts as well, whenever one has to deal
with the issues of stochastic analysis of complicated functionals of stochastic primitives.

\section*{Acknowledgements}
The authors would like to thank Dmitriy Katz for  stimulating discussions and the anonymous reviewers for providing constructive feedback. 
Research partially supported by   NSF grants DMI-0556106 and  CMMI-0726733. 

\bibliographystyle{amsalpha}
\bibliography{bibliography}


\section*{Appendix. Preliminary technical results}
In this section we establish some preliminary technical results. Using $\phi$ as defined by (\ref{eq:phi}), we let $U(x)=-ax+2b\phi(x)+c$ for some positive constants
$a,b,c$ satisfying
\begin{align}\label{eq:bovera}
{b\over a}\ge e^{2e}.
\end{align}

\begin{lemma}\label{lemma:U(x)isconvex}
$U(x)$ is strictly concave for $x\geq e^e$.
\end{lemma}
\begin{proof}
\beaa
\frac{\partial{U(x)}}{\partial{x}}&=&-a+b\sqrt{\frac{\ln\ln x}{x}}+\frac{b}{\ln{x}}\frac{1}{\sqrt{x\ln\ln x}}\\
\frac{\partial^2{U(x)}}{\partial{x}^2}&=&b
\Big(x^{-{1\over 2}}{1\over 2}(\ln\ln x)^{-{1\over 2}}\frac{1}{\ln x}\frac{1}{x}
+(\ln\ln x)^{\frac{1}{2}}(-{1\over 2}x^{-{3\over 2}})\Big) \\
&&+ b\Big(-(\ln x)^{-2}\frac{1}{x}(x\ln\ln x)^{-{1\over 2}}
+(\ln x)^{-1}(-{1\over 2})(x\ln\ln x)^{-{3\over 2}}(\frac{1}{\ln x}+\ln\ln{x})\Big)\\
&=&bx^{-{3\over 2}}({1\over 2})(\ln\ln x)^{-{1\over 2}}\Big(\frac{1}{\ln x}-(\ln\ln x)\Big) \\
&&+b\Big(-(\ln x)^{-2}\frac{1}{x}(x\ln\ln x)^{-{1\over 2}}\Big)+
b\Big((\ln x)^{-1}(-{1\over 2})(x\ln\ln x)^{-{3\over 2}}(\frac{1}{\ln x}+\ln\ln{x})\Big)\\
&<& 0 \text{~~since all three terms on RHS above are negative for $x\geq e^e$}
\eeaa
\end{proof}

\begin{lemma}\label{lemma:Unegative}
Assuming (\ref{eq:bovera}) and $e^e>(c/b)^2$,
\begin{align*}
U(x)<0 \qquad \text{$\forall ~x>(18b^2/a^2)\ln\ln(3b/a)$}.
\end{align*}
\end{lemma}

\begin{proof}
 Since $(18b^2/a^2)\ln\ln(3b/a)>e^e$, throughout the proof we restrict ourselves to the domain $x\geq e^e$. Since in addition $x>(c/b)^2$, we have $b\phi(x)\ge b\sqrt{x}>c$. In this range $-ax+2b\phi(x)+c\le -ax+3b\phi(x)=-ax+3b\sqrt{x\ln\ln x}$. This quantity is less than zero provided
\begin{align*}
\Big({x\over \ln\ln x}\Big)^{1\over 2}>{3b\over a}\triangleq\alpha.
\end{align*}
It is easy to check that $x/\ln\ln x$ is a strictly increasing function with
$\lim_{x\rightarrow\infty}(x/\ln\ln x)=\infty$. Let $x_0$ be the unique solution
of $x/\ln\ln x=\alpha^2$ on $x\geq e^e$.
We claim that $x_0\le 2\alpha^2\ln\ln\alpha$. The assertion of the lemma follows from this bound.
Let $A=2\alpha^2\ln\ln\alpha$. Then
\begin{align*}
{A\over \ln\ln A}&={2\alpha^2\ln\ln\alpha\over \ln(2\ln\alpha+\ln^{(3)}\alpha+\ln2)}\\
&\geq\frac{2\alpha^2\ln\ln{\alpha}}{\ln(4\ln\alpha)}\text{~~ since $\ln\alpha\geq\ln^{(3)}\alpha$ and $\ln\alpha>\ln{2}$}\\\
&\geq\frac{2\alpha^2\ln\ln{\alpha}}{2\ln(\ln\alpha)}
\text{~~ since $\ln\alpha>\ln(b/a)\ge 2e> 4$. }\\
&=\alpha^2.
\end{align*}
This implies $x_0\le A$ and the proof is complete.
\end{proof}



\begin{prop}\label{prop:Umax}
Under the assumption (\ref{eq:bovera})
\begin{align*}
\sup_{x\ge 0}U(x)\le 7(b^2/a)\ln\ln(b/a)+c.
\end{align*}
\end{prop}

\begin{proof}
Since $a>0$, then the supremum in $\sup_{x\ge 0}U(x)$ is achieved. Let $x^*$ be any value achieving $\max_{x\ge 0} U(x)$. First suppose $0\leq x^*<e^e$. It follows from the definition of $\phi$ in (\ref{eq:phi}) that $\phi(x^*)=1$ and thus $U(x^*)=-ax^*+2b+c$. Using $0\leq x^*<e^e$ and assumption (\ref{eq:bovera}), it is straightforward to check that $U(x^*)$ is indeed upper bounded from above by $7(b^2/a)\ln\ln(b/a)+c$. Next, we consider the case $x^*=e^e$, and using the fact that $a>0$, we obtain $U(x^*)\leq 2b\cdot\sqrt{e^e\ln\ln(e^e)}+c$. It is again straightforward to check that the aforementioned bound is upper bounded from above by $7(b^2/a)\ln\ln(b/a)+c$.

We now consider the case $x^*>e^e$. By Lemma~\ref{lemma:U(x)isconvex}, $x^*$ is the unique point satisfying $\frac{\partial{U(x^*)}}{\partial{x^*}}=0$, if it exists. The remainder of the proof is devoted to the final case where we obtain
\bea
0=\frac{\partial{U(x^*)}}{\partial{x^*}} =
-a+\frac{b(\frac{1}{\ln{x^*}}+\ln\ln{x^*})}{\sqrt{x^*\ln\ln{x^*}}}\label{eqn:deriv}
\eea
Continuing further, (\ref{eqn:deriv})
implies
\bea \frac{\sqrt{x^*\ln\ln{x^*}}}{\ln\ln{x^*} +
\frac{1}{\ln{x^*}}} = \frac{b}{a}\triangleq \alpha. \label{kstar}
\eea
Note
\beaa\label{x_bound}
&&\frac{x^*}{\ln\ln x^*}>\alpha^2\\
&&\frac{x^*}{2\ln\ln x^*}<\alpha^2 \text{~~since $\ln\ln
x^*>\frac{1}{\ln x^*}$ for $x \geq e^e$}
\eeaa
It is easy to check that $x/\ln\ln x$ is a strictly increasing function for $x\ge e^e$
and $\lim_{x\rightarrow\infty}(x/\ln\ln x)=\infty$.  (\ref{eq:bovera}) implies that
there exist  unique $x_{\min}$ and $x_{\max}$ satisfying
\beaa \frac{x_{\min}}{\ln\ln
x_{\min}}=\alpha^2 && \frac{x_{\max}}{2\ln\ln x_{\max}}=\alpha^2
\eeaa
The  monotonicity of $x/\ln\ln x$ implies $x_{\min}\leq x^*\leq x_{\max}$.
In order to complete the proof of the proposition, we will first state and prove Lemmas \ref{lemma:1subsection:TSC(RO)sqrtlnln} and \ref{lemma:4subsection:TSC(RO)sqrtlnln}.

\begin{lemma}\label{lemma:1subsection:TSC(RO)sqrtlnln}
$x_{\min}\geq\alpha^2\ln\ln\alpha$ and $x_{\max}\leq 4\alpha^2\ln\ln\alpha$.
\end{lemma}
\begin{proof} Let $B_1 = \alpha^2\ln\ln\alpha$. Then \beaa
\frac{B_1}{\ln\ln B_1} &=&\frac{\alpha^2\ln\ln{\alpha}}{\ln\ln(\alpha^2\ln\ln{\alpha})}\\
&<&\frac{\alpha^2\ln\ln{\alpha}}{\ln\ln\alpha} \text{ ~~since $\ln\ln\alpha \geq 1$ for $\alpha\ge e^{2e}$}\\
&=&\alpha^2.
\eeaa
Thus since $\frac{x}{\ln\ln x}$ is increasing for $x\geq e^e$, we have $x_{\min} \geq B_1 $ and the first assertion is established.

Let $B_2 = 4\alpha^2\ln\ln{\alpha}$. Then
\beaa
\frac{B_2}{2\ln\ln B_2}  &=&\frac{4\alpha^2\ln\ln{\alpha}}{2\ln\ln(4\alpha^2\ln\ln{\alpha})}\\
&=&\frac{4\alpha^2\ln\ln{\alpha}}{2\ln(2\ln\alpha+\ln^{(3)}\alpha+\ln4)}\\
&\geq&\frac{4\alpha^2\ln\ln{\alpha}}{2\ln(4\ln\alpha)}\text{~~ since $\ln\alpha\geq\ln^{(3)}\alpha$ and $\ln\alpha>\ln{4}$}\\\
&\geq&\frac{4\alpha^2\ln\ln{\alpha}}{4\ln(\ln\alpha)}
\text{~~ since $\ln\alpha\ge 2e> 4$. }\\
&=&\alpha^2.
\eeaa
Thus, again since $x/\ln\ln x$ is increasing for $x\ge e^e$, then the second assertion follows.
\end{proof}

Lemma~\ref{lemma:1subsection:TSC(RO)sqrtlnln} and $x_{\min}\leq x^* \leq x_{\max}$ imply
\bea\label{lemma:3subsection:TSC(RO)sqrtlnln}
\alpha^2\ln\ln{\alpha}\leq x^* \leq 4\alpha^2\ln\ln{\alpha}.
\eea

\begin{lemma}\label{lemma:4subsection:TSC(RO)sqrtlnln}
$\sqrt{x_{\max}\ln\ln{x_{\max}}} \leq 4\alpha\ln\ln{\alpha}$.
\end{lemma}
\begin{proof}
\beaa
\sqrt{x_{\max}\ln\ln{x_{\max}}} &\le
& \sqrt{\big(4\alpha^2\ln\ln{\alpha}\big)\ln\ln{\big(4\alpha^2\ln\ln{\alpha}\big)}}
\text{~~by Lemma~\ref{lemma:1subsection:TSC(RO)sqrtlnln}}\\
&=& \alpha\sqrt{4\ln\ln{\alpha}}\sqrt{\ln{\big(2\ln{\alpha}+\ln^{(3)}{\alpha}+\ln{4}\big)}}\\
&\leq& \alpha\sqrt{4\ln\ln{\alpha}}\sqrt{\ln{\big(4\ln{\alpha}\big)}}
\text{~~since $\ln\alpha\geq\ln^{(3)}\alpha$ and $\ln\alpha\geq\ln(e^{2e})>\ln{4}$}\\
&\leq& \alpha\sqrt{4\ln\ln{\alpha}}\sqrt{2\ln{\ln{\alpha}}}
\text{~~since $\ln\alpha>4$}
\eeaa
and the lemma follows from the last step.
\end{proof}

We now complete the proof of Proposition~\ref{prop:Umax}. We have
\beaa
U(x^*)&\leq & -ax^*+2b\sqrt{x^*\ln\ln{x^*}}+c\\
&\leq& -ax_{\min}+2b\sqrt{x_{\max}\ln\ln{x_{\max}}}+c\text{~~ since $x_{\min}\leq x^*\leq x_{\max}$}\\
&\leq& -ax_{\min}+8b\alpha\ln\ln\alpha+c
\text{~~by Lemma \ref{lemma:4subsection:TSC(RO)sqrtlnln}}\\
&\leq&-a\alpha^2\ln\ln{\alpha}
+8b\alpha\ln\ln\alpha+c
\text{~~by Lemma \ref{lemma:1subsection:TSC(RO)sqrtlnln}}\\
&=&7(b^2/a)\ln\ln(b/a)+c.
\eeaa
\end{proof}

\end{document}